\documentclass[12pt]{article}
\usepackage{amsmath,amssymb,amsthm}
\usepackage{graphicx}
\usepackage{booktabs}
\usepackage{cite}
\usepackage{url}
\usepackage{mathrsfs}
\usepackage{bm}
\usepackage{xcolor}

\usepackage[top=2.5cm, bottom=2.5cm, left=2.5cm, right=2.5cm]{geometry}

\newtheorem{theorem}{Theorem}[section]
\newtheorem{lemma}[theorem]{Lemma}
\newtheorem{proposition}[theorem]{Proposition}
\newtheorem{corollary}[theorem]{Corollary}
\newtheorem{definition}[theorem]{Definition}
\newtheorem{remark}[theorem]{Remark}

\usepackage{hyperref}
\title{Dimensionality reduction and width of deep neural networks based on topological degree theory}
\author{Xiao-Song Yang\\
	School of Mathematics and Statistics,\\
	Huazhong University of Science and Technology,\\
	Wuhan 430074, P. R. China\\
	\emph{Hubei Key Laboratory of Engineering Modeling and Scientific Computing},\\
	Huazhong University of Science and Technology, Wuhan 430074, China\\
	\texttt{yangxs@hust.edu.cn}}
\date{}

\begin{document}
	
	\maketitle
	
	\begin{abstract}
		In this paper we present a mathematical framework on linking of embeddings of compact topological spaces into Euclidean spaces and separability of linked embeddings under a specific class of dimension reduction maps. As applications of the established theory, we provide some fascinating insights into classification and approximation problems in deep learning theory in the setting of deep neural networks.
	\end{abstract}
	
	
	
	\section{Introduction}\label{Introduction}
	
	Dimensionality reduction (DR) and deep neural networks (DNNs) are two important aspects in data analysis. In data analysis and deep learning, the datasets are often high-dimensional and exhibit some complicated topological structures due to various backgrounds from science to engineering \cite{carlsson2009topology, fefferman2016testing, higham2019deep, hinton2006reducing, korman2018autoencoding, mohri2018foundations}. Traditional approaches to data analysis and visualization, in particular on images recognition, often fail in the high-dimensional setting, and a common practice is to perform dimensionality reduction \cite{fefferman2016testing, korman2018autoencoding, tenenbaum2000global} in order to make data analysis tractable and economic, and the DNNs is a powerful tool in dealing with non-linear dimensionality reduction problems.
	
	It has now been recognized that practical datasets often consists of features of low intrinsic dimensions with some nontrivial topological structures \cite{carlsson2009topology, fefferman2016testing, korman2018autoencoding}, and the geometric structure of datasets heavily affect the architecture of the deep neural networks. Nonetheless, how and to what extent the geometric (topological) structure of datasets is connected with the architecture of a deep neural network remains unclear and is an active research area of deep learning in recent years.
	
	In view of classification of dataset by machine learning, the following problems are of theoretical interests as well as practical significance.
	
	Let \(A_{i}\subset\mathbb{R}^{n},i=1,2,\cdots,k\) be connected compact sets that are mutually disjoint. Does there exist a linear map \(L:\mathbb{R}^{n}\rightarrow\mathbb{R}^{m}\), \(m\leq n\), such that \(L(A_{i})\cap L(A_{j})=\emptyset\) for all \(i\neq j\)?
	
	More generally, let \(A_{i}\subset\mathbb{R}^{n},i=1,2,\cdots,k\) be disjoint compact topological spaces, and \(f_{i}:A_{i}\rightarrow\mathbb{R}^{n},i=1,2,\cdots,k\) are embeddings satisfying \(f_{i}(A_{i})\cap f_{j}(A_{j})=\emptyset\). Under what conditions does a continuous (differentiable if possible) dimension reduction map \(F:\mathbb{R}^{n}\rightarrow\mathbb{R}^{m}\) (\(m\leq n\)) fail to have the property that \(F\circ f_{i}(A_{i})\cap F\circ f_{j}(A_{j})=\emptyset\) for all \(i\neq j\)? In addition, do there exist pair-wise disjoint balls \(D_{i}\subset\mathbb{R}^{n}\), so that \(f(A_{i})\subset D_{i},i=1,2,\cdots,k\)?
	
	A natural question is: What types of location of disjointed compacts in the ambient space are topological obstructions to the existence of dimension reduction maps having the above property?
	
	Here, by dimension reduction map we mean a continuous map from high dimensional Euclidean space to a lower dimensional Euclidean space.
	
	The questions above seem trivial without restriction to continuous maps in view of the famous Urysohn's Lemma \cite{runde2005taste}. However, if we confine the continuous maps within some reasonable classes (for instance, linear, or more generally, deep neural networks), then many problems arise. One important problem is the effect of topological structure of dataset on efficiency of DNNs in classifying datasets. Since depth and width are two important measures on complexity of DNNs, investigations of this problem will provide insight into the complexity of DNNs and in particular, explain why the \textbf{width} of DNNs is of paramount importance in DNNs, especially in autoencoder design, which involves the reducing the dimensionality of data with DNNs.
	
	The main results (in sections \ref{Inseparability of linked compact sets under a class of dimension reduction maps} and \ref{Applications to problems in deep learning}) can be roughly stated as follows.
	
	\textbf{Result 1} Assume that an activation function \(\mu:\mathbb{R}^{k}\rightarrow\mathbb{R}^{k}\) is group homotopic to a homeomorphism \(F:\mathbb{R}^{k}\to F(\mathbb{R}^{k})\subset\mathbb{R}^{k}\). If a DNN \(Net^{\mu}:\mathbb{R}^{n}\rightarrow\mathbb{R}^{m}\), \(n>m\) constructed with activation function \(\mu\) has its width not greater than \(n\), then it fails to classify two compacts \(A\) and \(B\) of \(\mathbb{R}^{n}\) provided \(A\) and \(B\) are a link pair with nonzero degree.
	
	Here the activation map \(\mu:\mathbb{R}^{k}\rightarrow\mathbb{R}^{k}\) is defined coordinate-wisely by a scalar activation function \(\mu:\mathbb{R}\rightarrow\mathbb{R}\).
	
	The implication of the above statement is that an autoencoder neural networks with activation maps above having neck shaped architecture can not tell two labelled datasets apart if the datasets located in input space in some topologically complicated way.
	
	\textbf{Result 2} Let \(D^{n}\subset\mathbb{R}^{n}\) be the standard unit ball with \(n>1\), then for every \(\delta>0\), there is a continuous function \(f:D^{n}\rightarrow\mathbb{R}\), such that for each \(Net^{\mu}:D^{n}\rightarrow\mathbb{R}\), with width not greater than \(n\) but with arbitrary depth, we have
	\[\|Net^{\mu}(x)-f(x)\|\geq\delta,\quad x\in D^{n}\]
	Provided that the activation function \(\mu:\mathbb{R}^{k}\rightarrow\mathbb{R}^{k}\) is group homotopic to a homeomorphism \(F:\mathbb{R}^{k}\to F(\mathbb{R}^{k})\subset\mathbb{R}^{k}\).
	
	Since almost currently used activation functions such as Relu activation, a-leaky Relu activation, Elu activation, and all sigmoid-type monotone activation are group homotopic to a homeomorphism, it is easy to see from \textbf{result 2} that any DNN with such an activation fails to approximate some continuous function if its width is not greater than the dimension of the input space. This implies that to effectively approximate continuous functions defined on a compact domain in \(\mathbb{R}^{n}\), the lower bound of width of DNNs is at least \(n+1\). In \cite{hanin2017approximating}, The authors obtained this bound only for Relu.
	
	The rest of this paper is organized as follows. In section \ref{Links and link pairs} we establish a theory for describing the entangle of two disjoint compact sets in some Euclidean space in terms of homotopy and links and provide some criteria for two compacts being a link pair. In section \ref{A special homotopy} we introduce the notion of group homotopy to study some properties of continuous maps that are group homotopic to identity or a homeomorphism. Section \ref{Inseparability of linked compact sets under a class of dimension reduction maps} characterize a class of dimension reduction maps that fail to separate linked compact sets.
	
	For reader's convenience, we try to make the arguments as elementary as possible. But the results obtained seem very exciting from both perspectives of theoretical and application oriented investigations! As pleasant applications of the established theory in this paper, we show how the shape (topological structure) of the underlying datasets affects the architecture of deep neural networks in classification task and approximation ability.
	
	\section{Links and link pairs}\label{Links and link pairs}
	
	In this section we will present some preliminaries for link theory of subsets in Euclidean space in terms of embeddings and we are mainly concerned about under what conditions two embeddings into a Euclidean space are linked or not in order to give a mathematical theory for understanding deep neural networks.
	
	\subsection{Link in terms of homotopy}
	
	Let us first give some notions of link in terms of homotopy.
	
	Let \(A\) and \(B\) are two compact (connected) topological spaces and assume that \(f:A\rightarrow\mathbb{R}^{n}\) and \(g:B\rightarrow\mathbb{R}^{n}\) are embeddings. Throughout this paper, unless otherwise stated, we assume that \(f(A)\cap g(B)=\emptyset\).
	
	\begin{definition}\label{def:2.1}
		If there exists a continuous map \(H:A\times[0,1]\rightarrow\mathbb{R}^{n}\) such that
		\[H(x,0)=f(x),\quad x\in A\]
		and \(H(A,1)\) is satisfying the following property: there are two disjoint domains \(D_{1}\) and \(D_{2}\) in \(\mathbb{R}^{n}\) homeomorphic to the standard unit ball in \(\mathbb{R}^{n}\) such that \(H(A,1)\subset D_{1}\), \(g(B)\subset D_{2}\) and
		\[H(A,t)\cap g(B)=\emptyset,\quad\forall t\in[0,1].\]
		Then \(f\) and \(g\) are said to be unlinked. Otherwise, \(f\) and \(g\) are said to be linked.
	\end{definition}
	
	More generally we can give the link notion as follows.
	
	\begin{definition}\label{def:2.2}
		Suppose that there is a homotopy
		\[H:(A\cup B)\times[0,1]\rightarrow\mathbb{R}^{n}\]
		such that \(H(x,y,0)=f(x)\cup g(y)\), \(x\in A\), \(y\in B\), and there are two compact balls \(D\subset\mathbb{R}^{n}\) and \(\bar{D}\subset\mathbb{R}^{n}\) satisfying \(D\cap\bar{D}=\emptyset\) such that
		\[H(A,1)\subset D\quad\text{and}\quad H(B,1)\subset\bar{D}\]
		In addition \(H(A,t)\cap H(B,t)=\emptyset\), \(\forall t\in[0,1]\) then \(f\) and \(g\) are said to be unlinked. Otherwise \(f\) and \(g\) are said to be linked.
	\end{definition}
	
	Equivalently, the notion of link of two disjoint compact set can also be given in the following form.
	
	\begin{definition}\label{def:2.3}
		Suppose that any homotopy
		\[H:(A\cup B)\times[0,1]\rightarrow\mathbb{R}^{n}\]
		with \(H(x,y,0)=f(x)\cup g(y)\), \(x\in A\), \(y\in B\), satisfying \(H(A,1)\subset D\), \(H(B,1)\subset\bar{D}\) and \(D\cap\bar{D}=\emptyset\) where \(D\subset\mathbb{R}^{n}\) and \(\bar{D}\subset\mathbb{R}^{n}\) are two balls, implies that
		\[H(A,t)\cap H(B,t)\neq\emptyset\quad\text{for some}\quad t\in[0,1],\]
		then \(f\) and \(g\) are said to be linked.
	\end{definition}
	
	This equivalent definition intuitively means that if \(f\) and \(g\) are linked, then in the process of deformation of \(f(A)\) and \(g(B)\) to the final locations meeting the above requirements, there must be a time \(t\in[0,1]\) such that the moving images of \(f(A)\) and \(g(B)\) contact each other at \(t\).
	
	For convenience, \(f(A)\) and \(g(B)\) are called link pair if they are linked. In the case that \(f\) and \(g\) are inclusions, we just say \(A\) and \(B\) are a link pair. The following fact is obvious.
	
	\begin{proposition}\label{prop:2.1}
		Suppose that there are continuous maps \(\phi:S^{k}\to A\) and \(\varphi:S^{h}\to B\) with \(k+h=n-1\) such that \(f\circ\phi:S^{k}\rightarrow\mathbb{R}^{n}\) and \(g\circ\varphi:S^{h}\rightarrow\mathbb{R}^{n}\) are linked, then \(f(A)\) and \(g(B)\) are a link pair.
	\end{proposition}
	
	\begin{proof}
		If not, then there is a homotopy as defined in Definition~\ref{def:2.1} such that \(H(A,1)\subset D_{1}\), \(g(B)\subset D_{2}\) with \(D_{1}\cap D_{2}=\emptyset\) and \(H(A,t)\cap g(B)=\emptyset\), for all \(t\in[0,1]\). In particular
		\[H(\phi(S^{k}),1)\subset D_{1}\quad\text{and}\quad g\circ\varphi(S^{h})\subset D_{2}\]
		and
		\[H(\phi(S^{k}),t)\cap g\circ\varphi(S^{h})=\emptyset,\quad\forall t\in[0,1].\]
		Therefore \(f\circ\phi:S^{k}\rightarrow\mathbb{R}^{n}\) and \(g\circ\varphi:S^{h}\rightarrow\mathbb{R}^{n}\) are not linked, a contradiction.
	\end{proof}
	
	By a standard elementary argument from differential topology (using one parameter diffeomorphism group generated from some ordinary differential equations) it is easy to prove the following fact.
	
	\begin{proposition}\label{prop:2.2}
		Let \(M\) and \(N\) are two compact topological spaces (smooth manifolds) with \(\dim M<n-1\) and \(\dim N<n-1\). Assume that \(f:M\rightarrow\mathbb{R}^{n}\) and \(g:N\rightarrow\mathbb{R}^{n}\) are smooth embeddings. If there is a homotopy \(H:M\times[0,1]\rightarrow\mathbb{R}^{n}\) satisfying
		\[H(x,0)=f(x),\quad H(M\times[0,1])\cap g(N)=\emptyset\]
		and there is an open domain \(D\) homeomorphic to the unit disk \(D^{n}=\{x\in\mathbb{R}^{n}:\|x\|<1\}\), such that \(H(M,1)\subset D\) and \(g(N)\cap D=\emptyset\), then \(f\) and \(g\) are not a link pair.
	\end{proposition}
	
	\section{A special homotopy}\label{A special homotopy}
	
	For the sequel, we define a special homotopy that has the following property.
	
	\begin{definition}\label{def:3.1}
		Let \(X\) be a topological space. If two maps \(f\), \(g:X\to X\) are said to be group homotopic if there exists a homotopy \(H:X\times[0,1]\to X\) such that
		\[H(x,0)=f(x),\quad H(x,1)=g(x),\quad\forall x\in X\]
		and for \(t_{1},t_{2}\in[0,1]\) with \(t_{1}+t_{2}\in[0,1]\), the following holds
		\[H(x,t_{1}+t_{2})=H(x,t_{2})\circ H(x,t_{1})=H(H(x,t_{1}),t_{2}).\]
		We coin this term because the additive operation of the group of real numbers can be transformed to the composition operation of maps.
	\end{definition}
	
	An example of group homotopy can be derived from the flow \(\Phi(x,t)\) generated by ordinary differential equations \(\dot{x}=v(x)\). Fix \(t\neq 0\), the map \(H:\mathbb{R}^{n}\times[0,1]\rightarrow\mathbb{R}^{n}\) defined by \(H(x,s)=\Phi(x,st)\) is a homotopy from the identity to the map \(\Phi(x,t)\).
	
	Now we have the following fact.
	
	\begin{proposition}\label{prop:3.1}
		Suppose that a continuous map \(F:\mathbb{R}^{n}\rightarrow\mathbb{R}^{n}\) is group homotopic to the identity, and two embeddings \(F:A\rightarrow\mathbb{R}^{n}\) and \(G:B\rightarrow\mathbb{R}^{n}\) with \(f(A)\cap g(B)=\emptyset\) are a link pair. Then, \(F\circ f\) and \(F\circ g\) are a link pair or have a nonempty intersection.
		\[F\circ f\cap F\circ g\neq\emptyset\]
	\end{proposition}
	
	\begin{proof}
		Let \(H:\mathbb{R}^{n}\times[0,1]\rightarrow\mathbb{R}^{n}\) be the group homotopy from the identity id to \(F\).
		
		On the one hand, suppose that
		\[H(f(A),t)\cap H(g(B),t)=\emptyset\quad\forall t\in[0,1].\]
		Since \(f(A)\) and \(g(B)\) are a link pair, it is easy to construct a homotopy to show that \(F\circ f(A)\) and \(F\circ g(B)\) are also a link pair. In fact, if this is not the case, then there is a homotopy as defined in Definition~\ref{def:2.2},
		\[\bar{H}:(A\cup B)\times[0,1]\rightarrow\mathbb{R}^{n}\]
		\[\tilde{H}(x,y,0)=F\circ f(x)\cup F\circ g(y),x\in A,y\in B,\]
		Now construct \(\hat{H}:(A\cup B)\times[0,1]\rightarrow\mathbb{R}^{n}\) by \(H\) and \(\bar{H}\):
		\[\hat{H}(x,y,t)=H(x,y,2t)\quad x\in A,y\in B,t\in[0,\frac{1}{2}]\]
		\[\hat{H}(x,y,t)=\bar{H}(x,y,2t-1)\quad x\in A,y\in B,t\in[\frac{1}{2},1]\]
		showing that \(f(A)\) and \(g(B)\) are not a link pair, contradicting the link pair assumption.
		
		On the other hand, if there exists a \(\bar{t}\in[0,1]\) such that \(H(f(A),\bar{t})\cap H(g(B),\bar{t})\neq\emptyset\), then by Definition~\ref{def:3.1}, taking \(\bar{t}=1-\bar{t}\) we have
		\begin{align*}
			F(f(A))\cap F(g(B)) &= H(f(A),1)\cap H(g(B),1) \\
			&= H(H(f(A),\bar{t}),\bar{t})\cap H(H(g(B),\bar{t}),\bar{t}) \\
			&\supset H(H(f(A),\bar{t})\cap H(g(B),\bar{t}),\bar{t}) \neq\emptyset
		\end{align*}
		Thus we complete the proof.
	\end{proof}
	
	The following fact is obvious.
	
	\begin{proposition}\label{prop:3.2}
		For any homeomorphism \(F:\mathbb{R}^{n}\rightarrow\mathbb{R}^{n}\), if two embeddings \(f:A\rightarrow\mathbb{R}^{n}\) and \(g:B\rightarrow\mathbb{R}^{n}\) are a link pair, then so are the embeddings \(F\circ f\) and \(F\circ g\).
	\end{proposition}
	
	\begin{remark}\label{rem:3.1}
		In fact, the statement in Proposition~\ref{prop:3.2} holds true as long as \(F:\mathbb{R}^{n}\to F(\mathbb{R}^{n})\subset\mathbb{R}^{n}\) is a homeomorphism.
	\end{remark}
	
	\section{Inseparability of linked compact sets under a class of dimension reduction maps}\label{Inseparability of linked compact sets under a class of dimension reduction maps}
	
	As previous sections, throughout this section we consider two embeddings \(f:A\rightarrow\mathbb{R}^{n}\) and \(g:B\rightarrow\mathbb{R}^{n}\) with \(f(A)\cap g(B)=\emptyset\), and assume that \(f\) and \(g\) are a link pair.
	
	\begin{definition}\label{def:4.1}
		Let \(A_{1},A_{2},\cdots,A_{k}\) be compact subsets of \(\mathbb{R}^{n}\). A continuous map \(F:\mathbb{R}^{n}\rightarrow\mathbb{R}^{m}\) is said to separate \(A_{1},A_{2},\cdots,A_{k}\), if there exist mutually disjoint embedded balls \(D_{1},D_{2},\cdots,D_{k}\) in \(\mathbb{R}^{m}\) such that \(F(A_{i})\subset D_{i}\), \(i=1,2,\cdots,k\).
	\end{definition}
	
	Now in comparison with Proposition~\ref{prop:3.1} we can have a more general statement.
	
	\begin{theorem}\label{thm:4.1}
		Consider a link pair of embeddings \(f:A\rightarrow\mathbb{R}^{n}\) and \(g:B\rightarrow\mathbb{R}^{n}\). If a continuous map \(F:\mathbb{R}^{n}\rightarrow\mathbb{R}^{n}\) is group homotopic to a homeomorphism, then it can not separate \(f(A)\) and \(g(B)\).
	\end{theorem}
	
	\begin{proof}
		Let \(H:\mathbb{R}^{n}\times[0,1]\rightarrow\mathbb{R}^{n}\) be the group homotopy from a homeomorphism \(\Pi:\mathbb{R}^{n}\rightarrow\mathbb{R}^{n}\) to the map \(F\). On the one hand, suppose that
		\[H(\Pi\circ f(A),t)\cap H(\Pi\circ g(B),t)\neq\emptyset\quad\forall t\in[0,1]\]
		Since \(f(A)\) and \(g(B)\) are a link pair, and homeomorphism preserves link of a link pair, \(\Pi\circ f(A)\) and \(\Pi\circ g(B)\) are still a link pair, this implies that \(F\circ f\) and \(F\circ g\) are a link pair in view of the arguments in proof of Proposition~\ref{prop:3.1}, therefore can not be separated by disjoint balls.
		
		On the other hand, if
		\[H(\Pi\circ f(A),t)\cap H(\Pi\circ g(B),t)=\emptyset\quad\text{for some}\quad t\in[0,1]\]
		Then the arguments in proof of Proposition~\ref{prop:3.1} can apply to show that
		\[H(\Pi\circ f(A),1)\cap H(\Pi\circ g(B),1)=F\circ f(A)\cap F\circ g(B)\neq\emptyset.\]
		The proof is thus completed.
	\end{proof}
	
	\begin{lemma}\label{lem:4.1}
		Let \(\mathbb{R}^{m}\subset\mathbb{R}^{n}\) be a subspace with \(m<n\), and \(P:\mathbb{R}^{n}\rightarrow\mathbb{R}^{m}\) be the projection map. Suppose that the map
		\[\xi_{fg}(x,y)=\frac{f(x)-g(y)}{\|f(x)-g(y)\|},\quad(x,y)\in A\times B\]
		has a nonzero degree, then \(P(f(A))\cap P(g(B))\neq\emptyset\).
	\end{lemma}
	
	\begin{proof}
		Without loss of generality, suppose \(\mathbb{R}^{m}=\{x=(x_{1},\cdots,x_{n})\in\mathbb{R}^{n}:\)\(x_{m+i}=0,i=1,\cdots,n-m\}\). For \(f(x)=(f_{1}(x),\cdots,f_{n}(x))\in\mathbb{R}^{n},x\in A\), \(g(y)=(g_{1}(y),\cdots,g_{n}(y))\in\mathbb{R}^{n},y\in B\), let
		\[f^{t}(x)=(f_{1}(x),\cdots,f_{m}(x),(1-t)f_{m+1}(x),\cdots,(1-t)f_{n}(x))\]
		and
		\[g^{t}(x)=(g_{1}(y),\cdots,g_{m}(y),(1-t)g_{m+1}(y),\cdots,(1-t)g_{n}(y))\]
		Suppose on the contrary, we have \(P(f(A))\cap P(g(B))=\emptyset\), then
		\[f^{1}(x)\neq g^{1}(y),\quad\forall(x,y)\in A\times B\]
		This implies that for each \(t\in[0,1]\),
		\[f^{t}(x)\neq g^{t}(y),\quad\forall(x,y)\in A\times B\]
		Now define a homotopy \(H:A\times B\times[0,1]\to\mathbb{R}^{n}\)
		\[H(x,y,t)=\xi^{t}_{fg}(x,y)=\frac{f^{t}(x)-g^{t}(y)}{\|f^{t}(x)-g^{t}(y)\|}\]
		It is easy to see that this map is well defined and satisfies
		\[H(x,y,0)=\xi_{fg}(x,y)=\frac{f(x)-g(y)}{\|f(x)-g(y)\|}\]
		and
		\[H((x,y),1)\in S^{m-1}\subset S^{n-1},\quad\forall(x,y)\in A\times B\]
		Showing that the degree of \(\xi^{1}_{fg}(x,y)\) is zero. Since the degree is homotopy invariant, we are led to a contradiction.
	\end{proof}
	
	\begin{theorem}\label{thm:4.2}
		Let \(\mathbb{R}^{m}\subset\mathbb{R}^{n}\) be a subspace with \(m<n\), and \(L:\mathbb{R}^{n}\to\mathbb{R}^{m}\) be a linear map. Suppose that the map
		\[\xi_{fg}(x,y)=\frac{f(x)-g(y)}{\|f(x)-g(y)\|},\quad(x,y)\in A\times B\]
		has a nonzero degree, then \(L(f(A))\cap L(g(B))\neq\emptyset\).
	\end{theorem}
	
	\begin{proof}
		Let \(K_{L}\) be the kernel of \(L\) with \(\dim K_{L}\leq n-m\), and \(K\) be the complement space orthogonal to \(K_{L}\), i.e., \({\mathbb{R}}^{n}=K\oplus K_{L}\). Let
		\[f(x)=(f_{K}(x),f_{K_{L}}(x))\in{\mathbb{R}}^{n},\mbox{with}\quad f_{K}(x)\in K\quad\mbox{and}\quad f_{K_{L}}(x)\in K_{L}\]
		And
		\[g(y)=(g_{K}(y),g_{K_{L}}(y))\in{\mathbb{R}}^{n},\mbox{with}\quad g_{K}(y)\in K\quad\mbox{and}\quad g_{K_{L}}(x)\in K_{L}\]
		It is easy to see that
		\[Lf(x)=Lf_{K}(x)\]
		and
		\[Lg(y)=Lg_{K}(y)\]
		Since \(f_{K}(A)\) and \(g_{K}(B)\) are images under the projection \(P:{\mathbb{R}}^{n}\to K\), it follows from Lemma~\ref{lem:4.1} that \(f_{K}(A)\cap g_{K}(B)\neq\emptyset\), therefore
		\[L(f_{K}(A))\cap L(g_{K}(B))\neq\emptyset\]
		thus
		\[L(f(A))\cap L(g(B))=L(f_{K}(A))\cap L(g_{K}(B))\neq\emptyset.\]
	\end{proof}
	
	To provide more
	
	\begin{definition}\label{def:4.2}
		Let \(F,G:{\mathbb{R}}^{n}\rightarrow{\mathbb{R}}^{m}\) be two continuous maps. \(F\) and \(G\) are said to be conjugate, if there exist homeomorphisms \(\phi:{\mathbb{R}^{n}}\rightarrow{\mathbb{R}^{n}}\) and \(\varphi:{\mathbb{R}^{m}}\rightarrow{\mathbb{R}^{m}}\) such that
		\[F=\varphi\circ G\circ\phi.\]
	\end{definition}
	
	As a consequence of Theorem~\ref{thm:4.2}, we have the following observation.
	
	\begin{theorem}\label{thm:4.3}
		If a continuous map \(F:{\mathbb{R}}^{n}\rightarrow{\mathbb{R}}^{m}\) (\(m<n\)) is conjugate with a linear map \(L:{\mathbb{R}^{n}}\rightarrow{\mathbb{R}^{m}}\), then \(F(f(A))\cap F(g(B))\neq\emptyset\) provided that the link number of \(f(A)\) and \(g(B)\) is not zero.
	\end{theorem}
	
	\begin{proof}
		The condition in the above theorem means that there exist two homeomorphisms as in Definition~\ref{def:4.2} such that
		\[F = \varphi \circ L \circ \phi.\]
		Since homeomorphism preserves link pair, \(\phi\circ f(A)\) and \(\phi\circ g(B)\) are still a link pair. Then in view of Theorem~\ref{thm:4.2}, we have
		\[L\phi\circ f(A)\cap L\phi\circ g(B)\neq\emptyset\]
		So
		\[F(f(A))\cap F(g(B))=\varphi\circ L\phi\circ f(A)\cap\varphi\circ L\phi\circ g(B)\neq\emptyset\]
		and we complete the proof.
	\end{proof}
	
	\section{Applications to problems in deep learning}\label{Applications to problems in deep learning}
	
	We first show in this section how the topological property of the input datasets affects the architecture of a deep neural network in classification task. Then, as a by-product, we will geometrically show why the minimal width of a deep neural network (DNN) should be greater than the dimension of input (data) space in the setting of approximation.
	
	In deep learning, one of the powerful tools in classification applications are the famous deep neural network. The two typical nonlinear activation functions are the rectified linear unit defined as
	\[Relu(x)=\max\{0,x\},\quad x\in\mathbb{R}\]
	and the sigmoid function defined as
	\[\sigma(x)=\frac{1}{1+e^{-x}},\quad x\in\mathbb{R}\]
	In the vector form, we have the maps \(Relu:\mathbb{R}^{n}\rightarrow\mathbb{R}^{n}\) and \(\sigma:\mathbb{R}^{n}\rightarrow\mathbb{R}^{n}\) with
	\[Relu(x)=(Relu(x_{1}),\cdots,Relu(x_{n})),\quad x=(x_{1},\ldots,x_{n})\in\mathbb{R}^{n}\]
	and
	\[\sigma(x)=(\sigma(x_{1}),\cdots,\sigma(x_{n})),\quad x=(x_{1},\ldots,x_{n})\in\mathbb{R}^{n}\]
	In DNNs, a function defined in a hidden layer can be written as
	\[F:\mathbb{R}^{n}\rightarrow\mathbb{R}^{m}\]
	\[F(x) = Relu\circ A(x)\]
	with ReLU activation, and
	\[F(x)=\sigma\circ A(x)\]
	with sigmoid activation, where
	\[A(X)=Wx+b,\quad W\in M_{n\times m},\quad b\in\mathbb{R}^{m}.\]
	
	First we prove the following fact about the map \(\sigma\).
	
	\begin{proposition}\label{prop:5.1}
		The map \(Relu\) is group homotopic to the identity map.
	\end{proposition}
	
	\begin{proof}
		Consider the following systems of ordinary differential equations
		\[\dot{x}=Relu(x)-x,\quad x=(x_{1},\ldots,x_{n})\in\mathbb{R}^{n}\]
		Note that the function \(Relu(x)\) is continuous with Lipschitz constant one, therefore the solution \(\phi(x,t)\) of the above system is well defined for each \(t\in[0,\infty)\). Accordingly, we have a homeomorphism \(\phi(x,t):\mathbb{R}^{n}\to\mathbb{R}^{n}\) for each fixed \(t\). It is easy to see that the solution \(\phi(x,t)\) has the following properties:
		\begin{enumerate}
			\item \(\phi(x,t)=x\), for all \(x\in\mathbb{R}_{+}^{n}=\{x:x_{i}\geq 0,i=1,\cdots,n\}\)
			\item \(\lim\limits_{t\to\infty}\phi(x,t)=0\), for all \(x\in-\mathbb{R}_{+}^{n}=\{x:x_{i}\leq 0,i=1,\cdots,n\}\)
			\item Let \(\phi(x,t)(\phi_{1}(x,t),\cdots,\phi_{n}(x,t))\), then \(\lim\limits_{t\to\infty}\phi_{i}(x,t)=x_{i}\) if \(x_{i}\geq 0\) and \(\lim\limits_{t\to\infty}\phi_{i}(x,t)=0\) if \(x_{i}\leq 0\).
		\end{enumerate}
		Now construct the homotopy
		\[H:\mathbb{R}^{n}\times[0,1]\to\mathbb{R}^{n}\]
		\[H(x,t)=\phi(x,\frac{t}{1-t})\]
		It is easy to see that this homotopy is continuous group homotopy and satisfies
		\[H(x,0)=\phi(x,0)=x,\quad H(x,1)=Relu(x)\]
		The proof is completed.
	\end{proof}
	
	In designing a DNN for classification problem, one needs to construct a series of maps, called layer map, so that the resultant map obtained with the composition of these maps can have some learning ability as one requires. We only recall the notions and notations necessary in the present paper, for details we refer the reader to \cite{fefferman2016testing, korman2018autoencoding}.
	
	A deep neural network is a continuous map \(Net:\mathbb{R}^{n}\rightarrow\mathbb{R}^{m}\), \(n>m\), that is a composition of a finite sequence of functions and is of the form
	\[Net^{R} =F_{k}\circ F_{k-1}\circ\cdots\circ F_{1}(x)
	=Relu\circ A_{k}\circ Relu\circ A_{k-1}\circ\cdots\circ Relu\circ A_{1}(x)\]
	\[F_{i}:\mathbb{R}^{n_{i}}\rightarrow\mathbb{R}^{n_{i+1}},i=1,\cdots,k,\quad\text{with}\quad n_{1}=n\quad\text{and}\quad n_{k+1}=m\]
	\[F_{i}(x)=Relu\circ A_{i}(x),\quad A_{i}(x)=W_{i}x+b_{i},\quad W_{i}\in M_{n_{i}\times m_{i}},\quad b_{i}\in\mathbb{R}^{m_{i}}\]
	with ReLU activation.
	
	Or
	\[Net^{\sigma} =F_{k}\circ F_{k-1}\circ\cdots\circ F(x)
	=\sigma\circ A_{k}\circ\sigma\circ A_{k-1}\circ\cdots\circ\sigma\circ A_{1}(x)\]
	\[F_{i}:\mathbb{R}^{n_{i}}\rightarrow\mathbb{R}^{n_{i+1}},i=1,\cdots,k\quad\text{with}\quad n_{1}=n\quad\text{and}\quad n_{k+1}=m\]
	With sigmoid activation.
	
	In terms of deep learning theory, for a DNN of the above form, \(k\) is called its depth and \(\max\{n_{i}:i=1\cdots,k\}\) is called its width.
	
	\begin{definition}\label{def:5.1}
		Let \(B_{1},\cdots,B_{k}\) be compact subsets of \(\mathbb{R}^{n}\). A continuous map \(F:\mathbb{R}^{n}\rightarrow\mathbb{R}^{m}\) is said to classify \(B_{1},\cdots,B_{k}\), if \(F(B_{i})\cap F(B_{j})\neq\emptyset\), \(i\neq j\), \(i,j=1,\cdots,k\).
	\end{definition}
	
	\begin{proposition}\label{prop:5.2}
		If two compacts \(A\) and \(B\) of \(\mathbb{R}^{n}\) are a link pair with nonzero degree, then any DNN \(Net^{R}:\mathbb{R}^{n}\rightarrow\mathbb{R}^{m}\), \(n>m\), with width not greater than \(n\) fails to classify \(A\) and \(B\).
	\end{proposition}
	
	\begin{proof}
		If \(n_{2}<n\), then in view of Theorem~\ref{thm:4.2}, we have \(W_{1}(A)\cap W_{1}(B)\neq\emptyset\) implying that \(F_{1}(A)\cap F_{1}(B)\neq\emptyset\), thus \(Net^{R}(A)\cap Net^{R}(B)\neq\emptyset\).
		
		If \(n_{2}=n\), and \(A_{1}\) is a homeomorphism (i.e., \(W_{1}\) is not singular), then \(F_{1}(A)\) and \(F_{1}(B)\) is still a link pair with nonzero degree or they have nonempty intersection in view of Theorem~\ref{thm:4.1} and Proposition~\ref{prop:5.1}. Since there must be an index \(1<i\leq k-1\) such that \(n_{i}<m\), it can be seen that \(Net^{R}(A)\cap Net^{R}(B)\neq\emptyset\).
	\end{proof}
	
	In the same manner we can prove the following statement in the setting of sigmoid functions.
	
	\begin{proposition}\label{prop:5.3}
		If two compacts \(A\) and \(B\) of \({\mathbb{R}}^{n}\) are a link pair with nonzero degree, then any DNN \(Net^{\sigma}:{\mathbb{R}}^{n}\rightarrow{\mathbb{R}}^{m}\) with width not greater than \(n\) fails to classify \(A\) and \(B\).
	\end{proposition}
	
	In fact, keep in mind the remark after Proposition~\ref{prop:3.2} and Theorem~\ref{thm:4.1}, we can prove a more general assertion in the same manner as that of proof of Proposition~\ref{prop:5.2}.
	
	\begin{theorem}\label{thm:5.1}
		If an activation function \(\mu:{\mathbb{R}}^{k}\rightarrow{\mathbb{R}}^{k}\) is group homotopic to a homeomorphism \(F:{\mathbb{R}}^{k}\to F({\mathbb{R}}^{k})\subset{\mathbb{R}}^{k}\), then any DNN \(Net^{\mu}:{\mathbb{R}}^{n}\rightarrow{\mathbb{R}}^{m}\), \(n>m\) with width not greater than \(n\) fails to classify two compacts \(A\) and \(B\) of \({\mathbb{R}}^{n}\) provided \(A\) and \(B\) are a link pair with nonzero degree.
	\end{theorem}
	
	\begin{remark}\label{rem:5.1}
		The result above shows that in classification problems, the width of a DNN is a more important factor in construction of DNNs if the input data has some kind of complicated topological property such as nontrivial link. In particular, in reducing the dimensionality of data with DNNs, one has to design a DNN with its width larger than the dimension of data space to avoid potential loss of information in the dataset.
	\end{remark}
	
	Understanding the connection between the approximation capacity of a deep neural networks and its architecture is a key to reveal the power of deep learning. In the following we will provide some results on approximation capacity of a deep neural networks with width bounded by the dimension of input space.
	
	\begin{corollary}\label{cor:5.1}
		Let \(D^{n}\subset\mathbb{R}^{n}\) be the standard unit ball, then for every \(\delta>0\), there is a continuous function \(f:D^{n}\rightarrow\mathbb{R}\), such that for each \(Net:D^{n}\rightarrow\mathbb{R}\), with width not greater than \(n\) but with arbitrary depth, we have
		\[\|Net^{R}(x)-f(x)\|\geq\delta,\quad x\in D^{n}.\]
	\end{corollary}
	
	\begin{proof}
		Let \(S^{n-1}\subset\mathbb{R}^{n}\) be the standard unit sphere, and \(O\) be the origin, both of them are closed sets in \(D^{n}\). Then by the well-known Urysohn's Lemma \cite{runde2005taste}, there is a continuous function \(f:D^{n}\to[0,1]\) such that \(f(O)=0\) and \(f(S^{n-1})=2\delta\). Note that the map
		\[\xi(x,O)=\frac{x-O}{\|x-O\|},\quad(x,O)\in S^{n-1}\times\{O\}\]
		has degree one, therefore \(O\) and \(S^{n-1}\) are a link pair, it follows from Theorem~\ref{thm:5.1} that \(Net^{R}:D^{n}\to\mathbb{R}\) with width not greater than \(n\) but with arbitrary depth satisfies
		\[Net(S^{n-1})\cap Net(O)\neq\emptyset\]
		Let \(x\in S^{n-1}\) be the point such that \(Net^{R}(x)=Net^{R}(O)\). Then \(\|Net^{R}(x)-f(x)\|<\delta\) implies that
		\[\|Net^{R}(O)-f(O)\|=\|Net^{R}(x)-f(O)\|\geq\delta,\]
		and vis versa.
		
		A similar statement concerning the deep neural networks with \emph{ReLU} activation was proved by means of analytical approach \cite{hanin2017approximating}.
	\end{proof}
	
	For the deep neural networks with sigmoid activation, we can have the same assertion in view of Theorem~\ref{thm:5.2}.
	
	\begin{corollary}\label{cor:5.2}
		Let \(D^{n}\subset\mathbb{R}^{n}\) be the standard unit ball with \(n>1\), then for every \(\delta>0\), there is a continuous function \(f:D^{n}\to\mathbb{R}\), such that for each \(Net^{\sigma}:D^{n}\to\mathbb{R}\), with width not greater than \(n\) but with arbitrary depth, we have
		\[\|Net^{\sigma}(x)-f(x)\|\geq\delta,\quad x\in D^{n}.\]
	\end{corollary}
	
	Likewise, we have the following assertion for more general activation function.
	
	\begin{theorem}\label{thm:5.2}
		Let \(D^{n}\subset\mathbb{R}^{n}\) be the standard unit ball with \(n>1\), then for every \(\delta>0\), there is a continuous function \(f:D^{n}\to\mathbb{R}\), such that for each \(Net^{\mu}:D^{n}\to\mathbb{R}\), with width not greater than \(n\) but with arbitrary depth, we have
		\[\|Net^{\mu}(x)-f(x)\|\geq\delta,\quad x\in D^{n}.\]
		Provided that the activation function \(\mu:\mathbb{R}^{k}\rightarrow\mathbb{R}^{k}\) is group homotopic to a homeomorphism \(F:\mathbb{R}^{k}\to F(\mathbb{R}^{k})\subset\mathbb{R}^{k}\).
	\end{theorem}
	
	Because the Relu activation, the-leaky Relu activation, the Elu activation, and all sigmoid-type monotone activations are group homotopic to a homeomorphism, the above assertion holds for every deep neural networks based on one of these activations.
	
	\section{A summary}\label{A summary}
	
	In this paper we have present a topological argument on separability of compact sets embedded in the Euclidian space under a specific class DNN dimension reduction maps and provide lower bound of width of DNNs constructed with a broad class of activation functions such as Relu activation, a-leaky Relu activation, Elu activation, and all sigmoid-type monotone activation are group homotopic to a homeomorphism. We hope the argument in this paper may spur further investigation in understanding the black-box of deep neural networks from geometric point of view.
	
	\bibliographystyle{plain}

\end{document}